\documentclass[11pt, reqno, oneside]{amsart}
\usepackage{amsmath,amsfonts,amssymb,amsthm}

\usepackage{enumerate}
\usepackage{esint}
\usepackage[pagebackref,colorlinks,citecolor=blue,linkcolor=blue]{hyperref}
\usepackage[dvipsnames]{xcolor}
\usepackage{mathtools}  

\usepackage{mathtools}
\usepackage{setspace}

\usepackage{comment}

\usepackage{geometry}
\numberwithin{equation}{section}

\theoremstyle{plain}
\newtheorem{theorem}[equation]{Theorem}
\newtheorem{corollary}[equation]{Corollary}
\newtheorem{lemma}[equation]{Lemma}
\newtheorem{proposition}[equation]{Proposition}

\theoremstyle{definition}
\newtheorem{definition}[equation]{Definition}

\newtheorem{remark}[equation]{Remark}

\newcommand{\R}{{\mathbb R}}
\newcommand{\N}{{\mathbb N}}

\newcommand{\Om}{\Omega}

\providecommand{\vint}[1]{\mathchoice
	{\mathop{\vrule width 5pt height 3 pt depth -2.5pt
			\kern -9pt \kern 1pt\intop}\nolimits_{\kern -5pt{#1}}}
	{\mathop{\vrule width 5pt height 3 pt depth -2.6pt
			\kern -6pt \intop}\nolimits_{\kern -3pt{#1}}}
	{\mathop{\vrule width 5pt height 3 pt depth -2.6pt
			\kern -6pt \intop}\nolimits_{\kern -3pt{#1}}}
	{\mathop{\vrule width 5pt height 3 pt depth -2.6pt
			\kern -6pt \intop}\nolimits_{\kern -3pt{#1}}}}

\newcommand{\eps}{\varepsilon}
\newcommand{\loc}{\mathrm{loc}}

\newcommand{\BV}{\mathrm{BV}}

\newcommand{\liploc}{\mathrm{Lip}_{\mathrm{loc}}}

\newcommand{\ch}{\text{\raise 1.3pt \hbox{$\chi$}\kern-0.2pt}}

\newcommand{\mres}{\mathbin{\vrule height 2ex depth 2.2pt width
		0.12ex\vrule height -0.3ex depth 2.2pt width .5ex}}

\DeclareMathOperator{\Mod}{Mod}

\DeclareMathOperator{\capa}{Cap}
\DeclareMathOperator{\rcapa}{cap}

\DeclareMathOperator{\diam}{diam}

\begin{document}
\title[]{Capacitary density and removable sets\\
	for Newton--Sobolev functions
	in metric spaces
}
\author{Panu Lahti}
\address{Panu Lahti,  Academy of Mathematics and Systems Science, Chinese Academy of Sciences,
	Beijing 100190, PR China, {\tt panulahti@amss.ac.cn}}


\begin{abstract}
	In a complete metric space equipped with a doubling measure and
	supporting a $(1,1)$--Poincar\'e inequality, we show that
	every set satisfying a suitable capacitary density condition is removable for 
	Newton--Sobolev functions.
\end{abstract}

\date{\today}
\maketitle

\section{Introduction}

Removable sets for Sobolev functions is a classical topic that has been widely studied over several decades,
see e.g. \cite{BBL,Hed,KaWu,JoSm,Kols,Nta,Vai,Wu}.
A closed set $A\subset \R^n$ with zero Lebesgue measure is said to be removable for
the Sobolev class $W^{1,p}$ if
$W^{1,p}(\R^n\setminus A)=W^{1,p}(\R^n)$
as sets. We always consider $1\le p<\infty$.
Removable sets have been characterized by various conditions, which
however tend to be difficult to apply in practice,
e.g. as null sets for condenser capacities; see the discussion
in Koskela \cite[p. 292]{Kos}.
Indeed, the paper \cite{Kos} gave a more concrete sufficient condition for removability, proving that
so-called $p$-porous subsets of a hyperplane are removable for $W^{1,p}$.
It was also shown there that removability is equivalent with $\R^n\setminus A$ supporting a
$(1,p)$--Poincar\'e inequality; we can express the latter by saying that $A$ is removable for the 
$(1,p)$--Poincar\'e inequality.

In Koskela--Shanmugalingam--Tuominen \cite{KST} this type of removability result
was extended to Ahlfors regular metric measure spaces $(X,d,\mu)$
supporting a $(1,p)$--Poincar\'e inequality, with $1<p<\infty$.
We will give definitions in Section \ref{sec:prelis}.
Given $t>1$, one says that a set $A\subset X$ is $t$--porous if for every $x\in A$ there
exist arbitrarily small
radii $r>0$ such that
\[
A\cap (B(x,tr)\setminus B(x,r))=\emptyset.
\]
In \cite[Theorem A]{KST}, compact $t$--porous sets for sufficiently large $t$
were shown to be removable for the $(1,p)$--Poincar\'e inequality, $1<p<\infty$.
In \cite{BaKo} and \cite{KaWu}, suitable porosity conditions were also shown to be sufficient for removability for quasiconformal mappings.
The above $t$-porosity condition is a very concrete, but also quite strong requirement,
as very large annuli are required to have
empty intersection with the set $A$.

In the current paper, we give the following removability result
for Newton--Sobolev functions $N^{1,p}$,
which are an extension of Sobolev functions to metric spaces. In this result, porosity
is replaced by a capacitary density condition, which is also a very concrete, but much less restrictive
assumption.

\begin{theorem}\label{thm:removability theorem}
	Suppose $(X,d,\mu)$ is a complete metric space equipped with the doubling measure $\mu$
	and supporting a $(1,1)$-Poincar\'e inequality.
	Suppose $A\subset X$ is such that
	\begin{equation}\label{eq:liminf intro}
	\liminf_{r\to 0}\frac{\rcapa_1(A\cap B(x,r),B(x,2r))}{\rcapa_1(B(x,r),B(x,2r))}
	<c_*\quad\textrm{for 1-q.e. }x\in X
	\end{equation}
	and for a constant $c_*>0$ depending only on the doubling constant and the constants in the Poincar\'e
	inequality.
	Then $A$ is removable for $N^{1,p}(X)$ for all $1\le p<\infty$,
	and $X\setminus A$ supports a $(1,1)$--Poincar\'e inequality.
\end{theorem}

We will observe that if $A$ satisfies the aforementioned $t$-porosity with large enough
$t$, then it also satisfies \eqref{eq:liminf intro},
so in particular Theorem \ref{thm:removability theorem} essentially extends the results of
\cite{KST} to the case $p=1$.
But of course \eqref{eq:liminf intro} is satisfied also by many non-porous sets.
We also prove a version of Theorem \ref{thm:removability theorem} for the more general class of
functions of bounded variation (BV) in Corollary \ref{cor:BV}.
In fact, the proof of Theorem \ref{thm:removability theorem} also relies on BV theory and specifically on a
new Federer-type characterization of sets of finite perimeter proved in \cite{L-newFed}.

\section{Notation and definitions}\label{sec:prelis}

In this section we introduce the basic notation, definitions,
and assumptions that are employed in the paper.

\subsection{Measure theory and Newton--Sobolev functions}
Throughout this paper, $(X,d,\mu)$ is a complete metric space that is equip\-ped
with a metric $d$ and a Borel regular outer measure $\mu$ that satisfies
a doubling property, meaning that
there exists a constant $C_d\ge 1$ such that
\[
0<\mu(B(x,2r))\le C_d\mu(B(x,r))<\infty
\]
for every ball $B(x,r):=\{y\in X\colon d(y,x)<r\}$, with $x\in X$ and $r>0$.
Since some conditions that we consider hold only for balls with radius less than $\diam X$,
for simplicity we assume that $\diam X>0$, that is, $X$ consists of at least $2$ points.
When a property holds outside a set of $\mu$-measure zero, we say that it holds at
almost every (a.e.) point.
By a measurable set we always mean a $\mu$-measurable set.
We say that $X$, or $\mu$, is Ahlfors $Q$-regular with $Q>1$ if there
exists a constant $C_A\ge 1$ such that
for all $x\in X$ and $0<r<2\diam X$, we have
\begin{equation}\label{eq:Ahlfors Q regular}
\frac{1}{C_A}r^Q \le \mu(B(x,r))\le C_A r^Q.
\end{equation}
We will never assume Ahlfors $Q$-regularity in this paper, but we give the definition for reference,
since it is assumed in certain results in the literature.

All functions defined on $X$ or its subsets will take values in $[-\infty,\infty]$.
A function $u$ defined on a measurable set $H\subset X$
is said to be in the class $L^1_{\loc}(H)$ if for every $x\in H$ there exists
$r>0$ such that $u\in L^1(B(x,r)\cap H)$.
Other local spaces of functions are defined analogously.

By a curve we mean a rectifiable continuous mapping from a compact interval of the real line into $X$.
The length of a curve $\gamma$
is denoted by $\ell_{\gamma}$. We will assume every curve $\gamma$ to be parametrized
by arc-length (see e.g. \cite[Theorem 3.2]{Haj}), so that the curve is $\gamma\colon [0,\ell_{\gamma}]\to X$.
A nonnegative Borel function $g$ on $X$ is said to be an upper gradient 
of a function $u$ in a set $H\subset X$ if for all nonconstant curves $\gamma$ in $H$, we have
\begin{equation}\label{eq:definition of upper gradient}
	|u(x)-u(y)|\le \int_{\gamma} g\,ds:=\int_0^{\ell_{\gamma}} g(\gamma(s))\,ds,
\end{equation}
where $x$ and $y$ are the end points of $\gamma$.
We interpret $\infty-\infty=\infty$ and $-\infty-(-\infty)=-\infty$.
We also express \eqref{eq:definition of upper gradient} by saying that the pair $(u,g)$ satisfies
the upper gradient inequality on the curve $\gamma$.
Upper gradients were originally introduced in \cite{HK}.

We always consider $1\le p <\infty$, with a heavy focus on the case $p=1$.
The $p$-modulus of a family of curves $\Gamma$ is defined by
\[
\Mod_{p}(\Gamma):=\inf\int_{X}\rho^p\, d\mu,
\]
where the infimum is taken over all nonnegative Borel functions $\rho$ on $X$
such that $\int_{\gamma}\rho\,ds\ge 1$ for every curve $\gamma\in\Gamma$.
A property is said to hold for $p$-almost every curve
if it fails only for a curve family with zero $p$-modulus. 
If $g$ is a nonnegative $\mu$-measurable function on a $\mu$-measurable set $H\subset X$ and
the pair $(u,g)$ satisfies
the upper gradient inequality on $p$-a.e. curve $\gamma$ in $H$,
then we say that $g$ is a $p$-weak upper gradient of $u$ in $H$.

Denote the characteristic function of a set $A\subset X$ by $\ch_A$.
If $N\subset X$ with $\mu(N)=0$, then by using the Borel function
$\rho:=\infty \ch_{N'}$ for a Borel set $N'\supset N$ with $\mu(N')=0$,
we find that
\begin{equation}\label{eq:mu null set and curves}
\Mod_p(\{\gamma\colon [0,\ell_{\gamma}]\to X\colon \mathcal L^1(\gamma^{-1}(N))>0\})=0.
\end{equation}

Given a $\mu$-measurable set $H\subset X$, we let
\[
\Vert u\Vert_{N^{1,p}(H)}:=\Vert u\Vert_{L^p(H)}+\inf \Vert g\Vert_{L^p(H)},
\]
where the infimum is taken over all $p$-weak upper gradients $g$ of $u$ in $H$.
Then we define the Newton-Sobolev space
\[
N^{1,p}(H):=\{u:\|u\|_{N^{1,p}(H)}<\infty\}.
\]
This space was first introduced in \cite{Shan}.
If $H$ is an open subset of $\R^n$, then $u\in L^p(H)$ is in the classical Sobolev space $W^{1,p}(H)$
if and only if a suitable pointwise representative of $u$ is in $N^{1,p}(H)$, and then the
quantity $\Vert u\Vert_{N^{1,p}(H)}$
agrees with the classical Sobolev norm, see \cite[Theorem A.2, Corollary A.4]{BB}.
We understand Newton-Sobolev functions to be defined at every $x\in H$,
because this is needed for upper gradients to makes sense.
It is known that for every $u\in N_{\loc}^{1,p}(H)$ there exists a minimal $p$-weak
upper gradient $g_{u}$ of $u$ in $H$, satisfying $g_{u}\le g$ 
a.e. in $H$ for every $p$-weak upper gradient $g\in L_{\loc}^{p}(H)$
of $u$ in $H$, see \cite[Theorem 2.25]{BB}.
Usually we use the notation $g_u$, but since there can also be a dependence on the set $H$,
we sometimes use the notation $g_{u,H}$.

For truncations $u_M:=\max\{-M,\min\{M,u\}\}$, we clearly have
\begin{equation}\label{eq:truncation}
g_{u_M}\le g_u\quad \textrm{a.e.}
\end{equation}
For two functions $u,v\in N_{\loc}^{1,p}(H)\cap L^{\infty}(H)$, for the minimal $p$-weak
upper gradients we have the Leibniz rule
\begin{equation}\label{eq:Leibniz rule}
	g_{uv}\le |u|g_v+|v|g_u\quad\textrm{a.e. in }H;
\end{equation}
see \cite[Theorem 2.15]{BB}.

For any open set $\Om\subset X$,
the space of Newton-Sobolev functions with zero boundary values is defined by
\[
	N_0^{1,p}(\Om):=\{u|_{\Om}:\,u\in N^{1,p}(X)\textrm{ with }u=0\textrm { in }X\setminus \Om\}.
\]
This space can be understood to be a subspace of $N^{1,p}(X)$.

For any set $A\subset X$ and $0<R<\infty$, the Hausdorff content
of codimension one is defined by
\begin{equation}\label{eq:Hausdorff content}
\mathcal{H}_{R}(A):=\inf\left\{ \sum_{j}
\frac{\mu(B(x_{j},r_{j}))}{r_{j}}:\,A\subset\bigcup_{j}B(x_{j},r_{j}),\,r_{j}\le R\right\},
\end{equation}
where we take the infimum over finite and countable coverings.
We also define this for $R=\infty$, and there we only require $r_j<\infty$.
The codimension one Hausdorff measure of $A\subset X$ is then defined by
\[
\mathcal{H}(A):=\lim_{R\rightarrow 0}\mathcal{H}_{R}(A).
\]

We will assume throughout the paper that $X$ supports a $(1,1)$-Poincar\'e inequality,
meaning that there exist constants $C_P\ge 1$ and $\lambda \ge 1$ such that for every
ball $B(x,r)$, every $u\in L^1(X)$,
and every upper gradient $g$ of $u$, we have
\begin{equation}\label{eq:poincare inequality}
\vint{B(x,r)}|u-u_{B(x,r)}|\, d\mu 
\le C_P r\vint{B(x,\lambda r)}g\,d\mu,
\end{equation}
where 
\[
u_{B(x,r)}:=\vint{B(x,r)}u\,d\mu :=\frac 1{\mu(B(x,r))}\int_{B(x,r)}u\,d\mu.
\]

Given a set $A\subset X$ with $\mu(A)=0$, we say that the space $X\setminus A=(X\setminus A,d,\mu)$
supports a $(1,p)$--Poincar\'e inequality if
there exist constants $C_P'\ge 1$ and $\lambda' \ge 1$ such that for every
$x\in X\setminus A$, $r>0$, $u\in L^1(X)$,
and every upper gradient $g$ of $u$ in $X\setminus A$, we have
\begin{equation}\label{eq:poincare inequality minus A}
	\vint{B(x,r)}|u-u_{B(x,r)}|\, d\mu 
	\le C_P' r \left(\vint{B(x,\lambda' r)}g^p\,d\mu\right)^{1/p}.
\end{equation}
Note that by H\"older's inequality, the $(1,1)$--Poincar\'e implies the $(1,p)$--Poincar\'e inequality
for every $1<p<\infty$.

\subsection{Functions of bounded variation}

Next we introduce functions
of bounded variation on metric spaces, following Miranda Jr.
\cite{M}.
Given an open set $\Om\subset X$ and a function $u\in L^1_{\loc}(\Om)$,
we define the total variation of $u$ in $\Om$ by
\[
	\Vert Du\Vert (\Om):=\inf\left\{\liminf_{i\to\infty}\int_\Om g_{u_i}\,d\mu\colon \, u_i\in N^{1,1}_{\loc}(\Om),\, u_i\to u\textrm{ in } L^1_{\loc}(\Om)\right\},
\]
where each $g_{u_i}$ is the minimal $1$-weak upper gradient of $u_i$ in $\Om$.
We say that a function $u\in L^1(\Om)$ is of bounded variation, 
and denote $u\in\BV(\Om)$, if $\|Du\|(\Om)<\infty$.
For an arbitrary set $A\subset X$, we define
\[
\Vert Du\Vert (A):=\inf\{\Vert Du\Vert (W)\colon A\subset W,\,W\subset X
\text{ is open}\}.
\]
In \cite{M}, pointwise Lipschitz constants were used in place of weak upper gradients, but the theory
can be developed similarly with either definition. In the literature, it
is sometimes also required that $u_i\in \liploc(\Om)$ instead of $u_i\in N^{1,1}_{\loc}(\Om)$,
but for us this does not make a difference, since functions in
$N^{1,1}_{\loc}(\Om)$ can be approximated by functions in $\liploc(\Om)$ in the $\Vert \cdot\Vert_{N^{1,1}(\Om)}$-seminorm,
see \cite[Theorem 5.47]{BB}.

If $u\in L^1_{\loc}(\Om)$ and $\Vert Du\Vert(\Omega)<\infty$,
then $\|Du\|$ is
a Borel regular outer measure on $\Omega$ by \cite[Theorem 3.4]{M}.
A $\mu$-measurable set $E\subset X$ is said to be of finite perimeter if $\|D\ch_E\|(X)<\infty$, where $\ch_E$ is the characteristic function of $E$.
The perimeter of $E$ in $\Omega$ is also denoted by
\[
P(E,\Omega):=\|D\ch_E\|(\Omega).
\]

Applying the Poincar\'e inequality \eqref{eq:poincare inequality} to sequences of approximating
$N^{1,1}_{\loc}$-functions in the definition of the total variation, we get
the following $\BV$ version:
for every ball $B(x,r)$ and every 
$u\in L^1_{\loc}(X)$, we have
\[
\int_{B(x,r)}|u-u_{B(x,r)}|\,d\mu
\le C_P r \Vert Du\Vert (B(x,\lambda r)).
\]
For a $\mu$-measurable set $E\subset X$, by considering the two cases
$(\ch_E)_{B(x,r)}\le 1/2$ and $(\ch_E)_{B(x,r)}\ge 1/2$, from the above we get
the relative isoperimetric inequality
\begin{equation}\label{eq:relative isoperimetric inequality}
	\min\{\mu(B(x,r)\cap E),\,\mu(B(x,r)\setminus E)\}\le 2 C_P rP(E,B(x,\lambda r)).
\end{equation}

We define the lower and upper densities of a set $E\subset X$ at a point $x\in X$ as follows:
\[
\theta_*(E,x):=\liminf_{r\to 0}\frac{\mu(B(x,r)\cap E)}{\mu(B(x,r))}
\quad\textrm{and}\quad 
\theta^*(E,x):=\limsup_{r\to 0}\frac{\mu(B(x,r)\cap E)}{\mu(B(x,r))}.
\]
The measure-theoretic interior of $E\subset X$ is defined by
\begin{equation}\label{eq:measure theoretic interior}
I_E:=
\left\{x\in X\colon\theta^*(X\setminus E,x)=0\right\},
\end{equation}
and the measure-theoretic exterior by
\[
O_E:=
\left\{x\in X\colon\theta^*(E,x)=0\right\}.
\]
The measure-theoretic boundary is defined as
\begin{equation}\label{eq:measure theoretic boundary}
\partial^{*}E:=\left\{x\in X\colon \theta^*(E,x)>0
\textrm{ and }
\theta^*(X\setminus E,x)>0\right\}.
\end{equation}
It is straightforward to show that these are all Borel sets;
note also that the space $X$ is always partitioned into the disjoint sets
$I_E$, $O_E$, and $\partial^*E$.
We also let
\[
	E_{b}:=\{x\in X\colon \theta_{*}(E,x)\ge b\},\quad b> 0.
\]
The \emph{strong boundary} $\Sigma_{b}E$, for $0<b\le 1/2$,
is defined as $\Sigma_{b}E:=E_{b}\cap (X\setminus E)_{b}$.

For an open set $\Omega\subset X$ and a $\mu$-measurable set $E\subset X$ with $P(E,\Omega)<\infty$, we know that
\begin{equation}\label{eq:measure theoretic and strong boundary}
	\mathcal H((\partial^*E\setminus \Sigma_{\gamma}E)\cap \Om)=0
\end{equation}
for a number $\gamma=\gamma(C_d,C_P,\lambda)>0$, see \cite[Theorem 5.3, Theorem 5.4]{A1}.
We then also know that for any Borel set $A\subset\Omega$,
\begin{equation}\label{eq:def of theta}
	P(E,A)=\int_{\partial^*E\cap A}\theta_E\,d\mathcal H
	=\int_{\Sigma_{\gamma}E\cap A}\theta_E\,d\mathcal H,
\end{equation}
where
$\theta_E\colon \Om\to [\alpha,C_d]$ with $\alpha=\alpha(C_d,C_P,\lambda)>0$,
see \cite[Theorem 5.3]{A1} 
and \cite[Theorem 4.6]{AMP}.
In particular, $P(E,\Om)<\infty$ implies that $\mathcal H(\partial^*E\cap\Om)<\infty$.
Federer's characterization of sets of finite perimeter states that the converse is also true.
That is, if $E\subset X$ is a $\mu$-measurable set such that $\mathcal H(\partial^*E\cap \Om)<\infty$,
then $P(E,\Om)<\infty$,
see \cite[Theorem 1.1]{L-Fedchar}.
See also Federer \cite[Section 4.5.11]{Fed} for the original Euclidean result.

The strong boundary can also be used to characterize sets of finite perimeter, as follows.

\begin{theorem}[{\cite[Theorem 1.1]{L-newFed}}]\label{thm:new Federer}
	Let $\Om\subset X$ be an open set and let $E\subset X$ be a $\mu$-measurable set
	with $\mathcal H(\Sigma_{\beta} E\cap \Om)<\infty$,
	where $0<\beta\le 1/2$ only depends on the doubling constant of the measure
	and the constants in the Poincar\'e inequality. Then $P(E,\Om)<\infty$.
\end{theorem}

Throughout this paper, we will use $\beta$ to denote the constant from this theorem;
we can assume that $\beta\le \gamma$.
Combining Theorem \ref{thm:new Federer} and \eqref{eq:def of theta}, we obtain that for every
open $\Om\subset X$ and $\mu$-measurable $E\subset X$, we have
\begin{equation}\label{eq:perimeter and strong boundary}
P(E,\Om)\le C_d\mathcal H(\Sigma_{\beta}E\cap\Om).
\end{equation}

For a function $u$ defined on an open set $\Om\subset X$,
we abbreviate super-level sets by
\[
\{u>t\}:=\{x\in \Om\colon u(x)>t\},\quad t\in\R.
\]
The following coarea formula is given in \cite[Proposition 4.2]{M}:
if $\Omega\subset X$ is open and $u\in L^1_{\loc}(\Omega)$, then
\begin{equation}\label{eq:coarea}
	\|Du\|(\Omega)=\int_{-\infty}^{\infty}P(\{u>t\},\Omega)\,dt.
\end{equation}
The integral should be understood as an upper integral; however if either side is finite, then
both sides are finite and the integrand is measurable.
In this case, \eqref{eq:coarea} moreover holds with $\Om$ replaced by any Borel set $A\subset \Om$.

\subsection{Capacities}
Recall that we always consider $1\le p<\infty$.
The $p$-capacity of a set $A\subset X$ is defined by
\[
\capa_p(A):=\inf \Vert u\Vert_{N^{1,p}(X)}^p,
\]
where the infimum is taken over all $u\in N^{1,p}(X)$ satisfying
$u\ge 1$ in $A$.
If a property holds outside a set of $p$-capacity zero, we say that it holds
$p$-quasieverywhere, abbreviated $p$-q.e.
	We say that a set $U\subset X$ is $p$-quasiopen
	if for every $\eps>0$ there exists an
	open set $G\subset X$ such that $\capa_p(G)<\eps$ and $U\cup G$ is open.
	By \cite[Remark 3.5]{ShanH}, if $U\subset X$ is $p$-quasiopen, then
	\begin{equation}\label{eq:path open}
	\textrm{for }p\textrm{-a.e. curve }\gamma\colon [0,\ell_{\gamma}]\to X
	\textrm{ we have that }\gamma^{-1}(U)\textrm{ is relatively open.}
	\end{equation}
Denoting the family of curves $\gamma\colon [0,\ell_{\gamma}]\to X$
intersecting a set $A\subset X$ by $\Gamma_A$,
by \cite[Proposition 1.48]{BB} we know that
\begin{equation}\label{eq:cap and mod}
	\textrm{if }\capa_p(A)=0,
	\quad\textrm{then}\quad
	\Mod_p(\Gamma_A)=0.
\end{equation}

The variational $p$-capacity of a set $A\subset W$
with respect to a bounded open set $W\subset X$ is defined by
\[
\rcapa_p(A,W):=\inf \int_X g_u^p \,d\mu,
\]
where the infimum is taken over all $u\in N_0^{1,p}(W)$ satisfying
$u\ge 1$ in $A$, and $g_u$ is the minimal $p$-weak upper gradient of $u$ (in $X$).
This is an outer capacity, meaning that whenever $A\subset \overline{A}\subset W$, we have
\[
\rcapa_p(A,W)=\inf_{\substack{V\textrm{ open}\\ A\subset V\subset W}}\rcapa_p(V,W),
\]
see \cite[Theorem 6.19(vii)]{BB}.
For basic properties satisfied by capacities, such as monotonicity and countable subadditivity,
see \cite{BB}.

If $H\subset X$ is measurable, then
\begin{equation}\label{eq:quasieverywhere equivalence class}
	\textrm{if }u\in N^{1,p}(H)\textrm{ and }v=u\textrm{ p-q.e. in }H,\textrm{ then }\ \Vert v-u\Vert_{N^{1,p}(H)}=0,
\end{equation}
see \cite[Proposition 1.61]{BB}.
We also know that
\begin{equation}\label{eq:ae to qe}
	\textrm{if }u,v\in N^{1,p}(H)\textrm{ and }v=u\textrm{ a.e. in }H,\textrm{ then } 
	g_v=g_u\textrm{ a.e. in }H,
\end{equation}
by \cite[Corollary 1.49, Proposition 1.59]{BB}.

For an open set $\Om\subset X$ and
$u\in N^{1,1}(\Om)$, we know that $1$-q.e. point is a Lebesgue point, that is,
\begin{equation}\label{eq:Lebesgue point result}
\lim_{r\to 0}\,\vint{B(x,r)}|u-u(x)|\,d\mu=0
\quad\textrm{for }1\textrm{-q.e. }x\in\Om;
\end{equation}
see \cite[Theorem 4.1]{KKST}
(in this paper $\mu(X)=\infty$ is assumed,
but this can be circumvented by using \cite[Lemma~3.1]{Maka}).

We will often consider the so-called precise representative
\begin{equation}\label{eq:precise representative}
	u^*(x):=\limsup_{r\to 0}\vint{B(x,r)}u\,d\mu.
\end{equation}
If $u\in\BV(\Om)$ and $\Vert Du\Vert$ is absolutely continuous with respect to $\mu$, then
\begin{equation}\label{eq:from BV to Sobolev}
	u^*\in N^{1,1}(\Om)\quad\textrm{with}\quad \int_{\Om}g_{u^*}\,d\mu \le C\Vert Du\Vert(\Om)
\end{equation}
for a constant $C\ge 1$ depending only on $C_d,C_P,\lambda$; 
by \cite[Remark 4.7]{HKLL} we know that this is true for \emph{some} representative of $u$, and then by
\eqref{eq:quasieverywhere equivalence class} and \eqref{eq:Lebesgue point result}
we know that the representative can be taken to be $u^*$.

For every $x\in X$ and $0<r<\tfrac 18 \diam X$, by \cite[Proposition 6.16]{BB} we have
\begin{equation}\label{eq:capa basic estimate}
C_{\mathrm{sob}}^{-1} \frac{\mu(B(x,r))}{r^p}\le \rcapa_p(B(x,r),B(x,2r))
\le C_d \frac{\mu(B(x,r))}{r^p},
\end{equation}
where $C_{\mathrm{sob}}\ge 1$ is a constant from a Sobolev inequality,
only depending on $C_d,C_P,\lambda$.
By \cite[Proposition 6.16]{BB}, for every $A\subset B(x,r)$ we then also have
\begin{equation}\label{eq:measure capa comparison}
	\begin{split}
\frac{\mu(A)}{\mu(B(x,r))}
&\le C_{\mathrm{sob}}r^p\frac{\rcapa_p(A,B(x,2r))}{\mu(B(x,r))}\\
&\le C_{\mathrm{sob}}C_d\frac{\rcapa_p(A,B(x,2r))}{\rcapa_p(B(x,r),B(x,2r))}
\quad\textrm{by }\eqref{eq:capa basic estimate}.
	\end{split}
\end{equation}

Next we define the fine topology in the case $p=1$.
For the analogous definition and theory in the case $1<p<\infty$, see
e.g. the monographs \cite{BB,MZ}.

\begin{definition}\label{def:1 fine topology}
	We say that $A\subset X$ is $1$-thin at the point $x\in X$ if
	\[
	\lim_{r\to 0}\frac{\rcapa_1(A\cap B(x,r),B(x,2r))}{\rcapa_1(B(x,r),B(x,2r))}=0.
	\]
	We say that a set $U\subset X$ is $1$-finely open if $X\setminus U$ is $1$-thin at every $x\in U$.
	Then we define the $1$-fine topology as the collection of $1$-finely open sets on $X$.
	
	We denote the $1$-fine closure of a set $H\subset X$,
	i.e. the smallest $1$-finely closed set containing $H$, by $\overline{H}^1$.
	The $1$-base $b_1 H$ is defined as the set of points
	where $H$ is \emph{not} $1$-thin.
\end{definition}

See \cite[Section 4]{L-FC} for discussion on this definition, and for a proof of the fact that the
$1$-fine topology is indeed a topology.
Using \eqref{eq:capa basic estimate}, we see that a set $A\subset X$ is $1$-thin at $x\in X$
if and only if
\[
\lim_{r\to 0}r\frac{\rcapa_1(A\cap B(x,r),B(x,2r))}
{\mu(B(x,r))}=0;
\]
in fact this was the formulation used in \cite{L-FC}.
From \eqref{eq:measure capa comparison} we get that if $U\subset X$ is a $1$-finely open set, then
\begin{equation}\label{eq:fine and measure topology}
	\textrm{for every }x\in U\ \textrm{ we have }\ \theta^*(X\setminus U,x)=0.
\end{equation}
By \cite[Theorem 4.6]{L-CK}, we know that
\begin{equation}\label{eq:A minus b 1 A}
	\capa_1(A\setminus b_1 A)=0.
\end{equation}
By \cite[Theorems 4.3 and 5.1]{HaKi}, for an arbitrary set $A\subset X$ we have
\begin{equation}\label{eq:capa and H measure}
	\capa_1(A)=0\quad\textrm{if and only if}\quad\mathcal H(A)=0.
\end{equation}
By \cite[Corollary 6.12]{L-CK} combined with \eqref{eq:capa and H measure}
we know that for an arbitrary set $U\subset X$,
\begin{equation}\label{eq:finely and quasiopen}
U\textrm{ is }1\textrm{-quasiopen }\ \Longleftrightarrow\ \ U=V\cup N\
\textrm{ where }V\textrm{ is 1-finely open and }\capa_1(N)=0.
\end{equation}
By \cite[Lemma 9.3]{BB15} we then know that every $1$-quasiopen and every $1$-finely open set
is measurable.
Moreover, by combining \eqref{eq:finely and quasiopen}
with \eqref{eq:path open} and \cite[Proposition 3.5]{BB15},
we have that if $H\subset X$ is $\mu$-measurable and
$U\subset X$ is $1$-finely open, and $u\in N^{1,1}_{\loc}(H)$,
then for the minimal $1$-weak upper gradients we have
\begin{equation}\label{eq:ug in finely open}
g_{u,H}=g_{u,H\cap U}\quad \textrm{a.e. in }H\cap U.
\end{equation}
By \cite[Proposition 3.3]{L-Fed}, we know that
if $A\subset W$ for an open set $W\subset X$, then
\begin{equation}\label{eq:capa of fine closure}
\capa_1(A)=\capa_1(\overline{A}^1)
\quad\textrm{and}\quad \rcapa_1(A,W)=\rcapa_1(\overline{A}^1\cap W,W).
\end{equation}

Our standing assumptions throughout the paper will be the following.\\

\emph{Throughout this paper we assume that $(X,d,\mu)$ is a complete metric space
	that is equipped with the doubling measure $\mu$ and supports a
	$(1,1)$-Poincar\'e inequality.}

\section{Preliminary results}

In this section we prove and record some preliminary results.

The following lemma is fairly standard, though we do not know a specific reference.

\begin{lemma}\label{lem:capacities}
Let $v\in L^1(X)$ be a pointwise defined function and let
$D_j\subset X$, $j\in\N$,
be measurable sets such that $\capa_1(D_j)\to 0$, and suppose that $v\in N^{1,1}(X\setminus D_j)$
with
\begin{equation}\label{eq:g Dj assumption}
\int_{X\setminus D_j}g_{v,X\setminus D_j}\,d\mu\le M<\infty\quad\textrm{for all }j\in\N.
\end{equation}
Then $v\in N^{1,1}(X)$ with $\int_{X}g_{v,X}\,d\mu\le M$.
\end{lemma}

Recall that $g_{v,X\setminus D_j}$ denotes the minimal $1$-weak upper gradient of
$v$ in $X\setminus D_j$.

\begin{proof}
By passing to a subsequence (not relabelled), we can assume that $\capa_1(D_j)\le 2^{-j}$.
Then by replacing the sets $D_j$ with $\bigcup_{k=j}^{\infty}D_k$, we can assume that
$D_{j+1}\subset D_j$.
From now on, we understand each $g_{v,X\setminus D_j}$ to be pointwise defined.
From the definition of minimal $1$-weak upper gradients, we get
$g_{v,X\setminus D_j}\le g_{v,X\setminus D_{j+1}}$ a.e.
in $X\setminus D_j$.
We  define pointwise
$g:=\limsup_{j\to \infty}g_{v,X\setminus D_j}$
in $X\setminus \bigcap_{j=1}^{\infty}D_j$, and then
$\int_{X}g\,d\mu\le M$ by \eqref{eq:g Dj assumption};
note that $\capa_1\left(\bigcap_{j=1}^{\infty}D_j\right)=0$ and thus
also $\mu\left(\bigcap_{j=1}^{\infty}D_j\right)=0$. Moreover, if we denote by
$\Gamma$ the family of curves intersecting $\bigcap_{j=1}^{\infty}D_j$, we have
$\Mod_1(\Gamma)=0$ by \eqref{eq:cap and mod}.
Note also that $g\ge g_{v,X\setminus D_{j}}$ a.e. in $X\setminus D_{j}$, for all $j\in\N$.
The family of curves $\gamma\colon [0,\ell_{\gamma}] \to X\setminus D_j$
for which the pair $(v,g_{v,X\setminus D_j})$
does not satisfy the upper gradient inequality on $\gamma$, and the family of curves
$\gamma\colon [0,\ell_{\gamma}] \to X\setminus D_j$ for which
\[
\mathcal L^1(\gamma^{-1}(\{g< g_{v,X\setminus D_{j}}\}))>0,
\]
are $\Mod_1$-negligible for all $j\in\N$, by \eqref{eq:mu null set and curves}.
Consider a curve $\gamma\colon [0,\ell_{\gamma}]\to X\setminus \bigcap_{j=1}^{\infty}D_j$
outside these exceptional families.
Since $D_j$ is a decreasing sequence,
$\gamma$ is in $X\setminus D_j$ for some $j\in\N$.
Hence we get
\[
|v(\gamma(0))-v(\gamma(\ell_{\gamma}))|
\le \int_{\gamma}g_{v,X\setminus D_j}\,ds
\le  \int_{\gamma}g\,ds.
\]
Thus $g$ is a $1$-weak upper gradient of $v$ in $X$, and we have the result.
\end{proof}

The following lemma is given in \cite[Lemma 3.3]{LahQBV}.
\begin{lemma}\label{lem:capacity and Newtonian function}
	Let $G\subset X$ with $\capa_1(G)<\infty$ and let $\eps>0$.
	Then there exists an open set $W\supset G$ with $\capa_1(W)\le C(\capa_1(G)+\eps)$ and a
	function $\eta\in N^{1,1}_0(W)$ with $0\le\eta\le 1$ on $X$, $\eta=1$ in $G$, and
	$\Vert \eta\Vert_{N^{1,1}(X)}\le C(\capa_1(G)+\eps)$,
	for a constant $C$ depending only on $C_d,C_P,\lambda$.
\end{lemma}

Next we record an isoperimetric inequality.
By \cite[Corollary 3.8]{BB}, there exist constants $C>0$ and $\sigma>0$ depending
only on $C_d$ such that for all $x\in X$ and $0<r< R< \diam X/2$, we have
\begin{equation}\label{eq:upper bound with sigma}
\frac{\mu(B(x,r))}{\mu(B(x,R))} \le C\left(\frac{r}{R}\right)^{\sigma}.
\end{equation}
Thus there exists a constant $L> 2$ depending only on $C_d$ such that
for every $x\in X$ and every $0<r< \diam X/(2L)$, we have
\[
\mu(B(x,2r))\le \frac{1}{2}\mu(B(x,Lr)).
\]
Then for any measurable set $H\subset B(x,2r)$,
by the relative isoperimetric inequality \eqref{eq:relative isoperimetric inequality} we have
\begin{equation}\label{eq:rel isop ineq application}
	\begin{split}
		\mu(H)
		&=\mu(H\cap B(x,Lr))\\
		&\le 2C_PL r P(H,B(x,L\lambda r))\\
		&\le 2C_PL r P(H,X).
	\end{split}
\end{equation}

Next we note that the following consequence of the Borel regularity of $\mu$ holds.
This can be proved similarly as in e.g. \cite[Lemma 4.3]{L-CK}.

\begin{lemma}\label{lem:weak Borel regularity}
	Let $A\subset X$. Then there exists a Borel set $A^*\supset A$ such that
	\[
	\mu(A\cap B(x,r))= \mu(A^*\cap B(x,r))
	\]
	for every ball $B(x,r)\subset X$.
\end{lemma}

Denote by $\lceil b\rceil$ the smallest integer at least $b\in\R$.
Let
\begin{equation}\label{eq:c star definition}
	c_*:=\frac{\alpha \beta}{8C_d^{2+\lceil \log_2 \lambda \rceil} C_P C_{\textrm{sob}} L},
\end{equation}
where $\alpha$ is defined after \eqref{eq:def of theta}
and can be assumed to take values in $(0,1)$, $0<\beta\le 1/2$ is from
Theorem \ref{thm:new Federer}, and $L\ge 2$ is from before \eqref{eq:rel isop ineq application}.
Recall that we also assume $C_P,C_{\textrm{sob}}\ge 1$ (for convenience).
The number $0<c_*<1$ is the constant used in our main Theorem \ref{thm:removability theorem}.

The removable sets that we consider are expected to have zero $\mu$-measure.
For this reason, the following basic fact is useful.
Recall that if a property holds outside a set of $1$-capacity zero, we say that it holds at
$1$-quasi every point, abbreviated ``$1$-q.e.''

\begin{lemma}\label{lem:zero measure}
Suppose $A\subset X$ is such that
\[
\liminf_{r\to 0}\frac{\rcapa_1(A\cap B(x,r),B(x,2r))}{\rcapa_1(B(x,r),B(x,2r))}<c_*
\quad\textrm{for 1-q.e. }x\in X.
\]
Then $\mu(A)=0$.
\end{lemma}
\begin{proof}
By \eqref{eq:measure capa comparison}, we have
\[
\liminf_{r\to 0}\frac{\mu(A\cap B(x,r))}{\mu(B(x,r))}
<C_{\textrm{sob}}C_dc_*
= \frac{\alpha \beta}{8C_d^{1+\lceil \log_2 \lambda \rceil} C_P L}
<1
\quad\textrm{for 1-q.e. }x\in X.
\]
By Lemma \ref{lem:weak Borel regularity}, we find a Borel set $A^*\supset A$ such that
\[
\mu(A\cap B(x,r))= \mu(A^*\cap B(x,r))
\]
for every ball $B(x,r)\subset X$.
Thus in fact
\[
\liminf_{r\to 0}\frac{\mu(A^*\cap B(x,r))}{\mu(B(x,r))}<1
\]
for $1$-q.e. $x\in X$, and then also for a.e. $x\in X$.
By the Lebesgue differentiation theorem, see e.g. \cite[p. 77]{HKSTbook}, it follows that necessarily
$\mu(A^*)=0$ and then also $\mu(A)=0$.
\end{proof}

The following standard fact can be proved as in e.g.
\cite[Lemma 2.6]{KKST-P}; note that when $D\subset X$ is Borel, the restriction
\[
\mathcal H \mres D(A):=\mathcal H(A\cap D),\quad A\subset X,
\]
is a Borel regular outer measure by \cite[Lemma 3.3.13]{HKSTbook}.

\begin{lemma}\label{lem:density}
Let $D\subset X$ be a Borel set
with $\mathcal H(D)<\infty$. Then for $\mathcal H$-a.e. $x\in X\setminus D$, we have
\[
\lim_{r\to 0}r\frac{\mathcal H(D\cap B(x,r))}{\mu(B(x,r))}=0.
\]
\end{lemma}

\section{Proof of Theorem \ref{thm:removability theorem}}\label{sec:removability}

In this section we prove the removability Theorem \ref{thm:removability theorem},
as well as a version for BV functions.

We define removable sets for Newton-Sobolev and BV functions as follows.
As usual, we consider $1\le p<\infty$.

\begin{definition}
	Let $A\subset X$ such that $\mu(A)=0$.
	We say that $A$ is removable for $N^{1,p}(X)$ if for
	every $u\in N^{1,p}(X\setminus A)$ there exists $v\in N^{1,p}(X)$ such that $v=u$ a.e. in
	$X$. We say that $A$ is isometrically removable for $N^{1,p}(X)$
	if moreover
	$g_v=g_u$ a.e. in $X$ for every such $v$.
\end{definition}

Note that more precisely, above we denote the minimal $p$-weak upper gradient $g_u=g_{u,X\setminus A}$.
If $g_v=g_u$ a.e. in $X$ for \emph{some} $v$ as above, then it is in fact true for \emph{every}
such $v$, by \eqref{eq:ae to qe}.
Note also that always $g_v\ge g_{v,X\setminus A} =g_{u,X\setminus A}$ a.e. by \eqref{eq:ae to qe}.

\begin{definition}
	Let $A\subset X$ be closed such that $\mu(A)=0$.
	We say that $A$ is removable for $\BV(X)$ if for
	every $u\in \BV(X\setminus A)$ we have $u\in\BV(X)$.
	We say that $A$ is isometrically removable for $\BV(X)$
	if moreover
	$\Vert Du\Vert(X)=\Vert Du\Vert(X\setminus A)$.
\end{definition}

Note that we understand BV functions to be defined only almost everywhere, and so it is not necessary
to talk about an extension $v$ here.
Moreover, note that we assume $A$ to be closed, because $\BV$ functions are defined only in open sets.

The cornerstone of our removability results is the following proposition, which essentially
says that $A$ is removable for sets of finite perimeter.

\begin{proposition}\label{prop:finite perimeter removability}
	Let $A\subset X$ be a closed set such that
	\begin{equation}\label{eq:H a.e. condition}
	\liminf_{r\to 0}\frac{\rcapa_1(A\cap B(x,r),B(x,2r))}{\rcapa_1(B(x,r),B(x,2r))}<c_*
	\quad\textrm{for 1-q.e. }x\in A.
	\end{equation}
	Let $E\subset X$ be a $\mu$-measurable set with
	$\mathcal H(\partial^* E\setminus A)<\infty$.
	Then $\mathcal H(\Sigma_{\beta}E\cap A)=0$.
\end{proposition}
\begin{proof}
	The inequality in \eqref{eq:H a.e. condition} holds outside a set $N\subset A$ with $\capa_1(N)=0$,
	and thus also $\mathcal H(N)=0$ by \eqref{eq:capa and H measure}.
Fix $x\in \Sigma_{\beta}E\cap A\setminus N$ (assuming it exists).
There is a sequence $r_j\searrow 0$ such that
\begin{equation}\label{eq:choice of rj}
\limsup_{j\to\infty}\frac{\rcapa_1(A\cap B(x,r_j),B(x,2r_j))}{\rcapa_1(B(x,r_j),B(x,2r_j))}<c_*.
\end{equation}
By the definition of the variational $1$-capacity, and recalling that it is an outer capacity,
for every $r>0$ we find a function $v_r\in N^{1,1}_0(B(x,2r))$ such that
$v_r\ge 1$ in a neighborhood of $A\cap B(x,r)$, and
\[
\int_X g_{v_r}\,d\mu\le \rcapa_1(A\cap B(x,r),B(x,2r))+\mu(B(x,r)).
\]
Here $v_r\in N^{1,1}(X)\subset  \BV(X)$ with $\Vert Dv_r\Vert(X)\le \int_X g_{v_r}\,d\mu$,
and then by the coarea formula \eqref{eq:coarea} we find
a set $A(x,r):=\{v_r>t\}$\label{Axr}
for some $0<t<1$, for which
\begin{equation}\label{eq:Axr choice}
	P(A(x,r),X)\le \Vert Dv_r\Vert(X) \le \rcapa_1(A\cap B(x,r),B(x,2r))+\mu(B(x,r)),
\end{equation}
and also $A(x,r)\subset B(x,2r)$, and $A\cap B(x,r)$ is contained in the interior
of $A(x,r)$.
By \eqref{eq:def of theta},
\begin{equation}\label{eq:perimeter small estimate}
	\begin{split}
		\limsup_{j\to\infty}r_j\frac{\mathcal H(\partial^* A(x,r_j))}{\mu(B(x,r_j))}
		&\le \frac{1}{\alpha}\limsup_{j\to\infty}r_j\frac{P(A(x,r_j),X)}{\mu(B(x,r_j))}\\
		&\le \frac{1}{\alpha}\limsup_{j\to\infty}r_j\frac{\rcapa_1(A\cap B(x,r_j),B(x,2r_j))}{\mu(B(x,r_j))}
		\quad\textrm{by }\eqref{eq:Axr choice}\\
		&\le \frac{C_d}{\alpha}\limsup_{j\to\infty}\frac{\rcapa_1(A\cap B(x,r_j),B(x,2r_j))}
		{\rcapa_1(B(x,r_j),B(x,2r_j))}\quad\textrm{by }\eqref{eq:capa basic estimate}\\
		&\le \frac{C_d c_*}{\alpha}
		\quad \textrm{by }\eqref{eq:choice of rj}\\
		&\le \frac{\beta}{8 C_P C_d^{\lceil\log_2 \lambda\rceil}}\quad \textrm{by }\eqref{eq:c star definition}.
	\end{split}
\end{equation}
By the doubling property of $\mu$,
\begin{align*}
	\limsup_{j\to \infty}\frac{\mu( A(x, r_j))}{\mu(B(x,r_j/\lambda))}
	& \le C_d^{\lceil\log_2 \lambda\rceil}\limsup_{j\to \infty}\frac{\mu(A(x, r_j))}{\mu(B(x,r_j))}\\
	& \le 2C_P L C_d^{\lceil\log_2 \lambda\rceil}\limsup_{j\to \infty}r_j\frac{P(A(x,r_j),X)}{\mu(B(x,r_j))}
	\quad\textrm{by }\eqref{eq:rel isop ineq application}\\
	& \le 2C_P  L C_d^{1+\lceil\log_2 \lambda\rceil} c_*
	\quad\textrm{by }\eqref{eq:perimeter small estimate}\textrm{ middle 3 inequalities}\\
	& \le \beta/2\quad \textrm{by }\eqref{eq:c star definition}.
\end{align*}
Combining this with the fact that $x\in \Sigma_{\beta}E$, we get
\begin{align*}
\liminf_{j\to\infty}\frac{\mu(B(x,r_j/\lambda)\setminus (E\cup A(x,r_j)))}{\mu(B(x,r_j))}
&\ge \frac{1}{C_d^{\lceil\log_2 \lambda\rceil}}\liminf_{j\to\infty}\frac{\mu(B(x,r_j/\lambda)
	\setminus (E\cup A(x,r_j)))}{\mu(B(x,r_j/\lambda ))}\\
&\ge \frac{\beta}{2C_d^{\lceil\log_2 \lambda\rceil}},
\end{align*}
and using again the fact that $x\in \Sigma_{\beta}E$, we have in total
\begin{equation}\label{eq:beta over two}
	\begin{split}
	&\liminf_{j\to\infty}\frac{\mu(B(x,r_j/\lambda)\cap (E\cup A(x,r_j)))}{\mu(B(x,r_j))}
	\ge \frac{\beta}{C_d^{\lceil\log_2 \lambda\rceil}}\quad \textrm{and}\\
	&\qquad \qquad \qquad \liminf_{j\to\infty}\frac{\mu( B(x,r_j/\lambda)\setminus (E\cup A(x,r_j)))}{\mu(B(x,r_j))}
	\ge \frac{\beta}{2C_d^{\lceil\log_2 \lambda\rceil}}.
	\end{split}
\end{equation}
Thus by \eqref{eq:relative isoperimetric inequality} and \eqref{eq:def of theta},
\begin{equation}\label{eq:E cup A boundary}
\begin{split}
	\liminf_{j\to\infty}r_j\frac{\mathcal H(\partial^*(E\cup A(x, r_j))\cap B(x,r_j))}{\mu(B(x,r_j))}
	\ge \frac{ \beta}{4 C_P C_d^{\lceil\log_2 \lambda\rceil}}.
\end{split}
\end{equation}
From the definition of the measure-theoretic interior and boundary
\eqref{eq:measure theoretic interior}, \eqref{eq:measure theoretic boundary},
it is straightforward to verify that
\begin{align*}
	\partial^*(E\cup A(x,r_j))\cap B(x, r_j)
	&\subset \Big[(\partial^*E\setminus I_{A(x,r_j)})\cup\partial^*A(x,r_j)\Big]\cap B(x,r_j )\\
	&\subset \Big[(\partial^*E\setminus A)\cup\partial^*A(x,r_j)\Big]\cap B(x, r_j).
\end{align*}
Thus
\begin{align*}
&\liminf_{j\to\infty}r_j\frac{\mathcal H((\partial^*E\setminus A)\cap B(x,r_j))}{\mu(B(x,r_j))}\\
&\qquad \ge \liminf_{j\to\infty}r_j\frac{\mathcal H(\partial^*(E\cup A(x,r_j))\cap B(x, r_j))}
{\mu(B(x,r_j))}-
\limsup_{j\to\infty}r_j\frac{\mathcal H(\partial^*A(x,r_j)\cap B(x,r_j))}{\mu(B(x,r_j))}\\
&\qquad\ge \frac{ \beta}{8 C_P C_d^{\lceil\log_2 \lambda\rceil}}>0
\quad\textrm{by }\eqref{eq:perimeter small estimate},\eqref{eq:E cup A boundary}.
\end{align*}
Recall that we have considered an arbitrary point $x\in \Sigma_{\beta}E\cap A\setminus N$.
However, since $\partial^*E\setminus A$ is a Borel set, by Lemma \ref{lem:density} we have that
\[
\lim_{r\to 0}r\frac{\mathcal H((\partial^*E\setminus A)\cap B(x,r))}{\mu(B(x,r))}=0
\]
for $\mathcal H$-a.e. $x\in X\setminus (\partial^*E\setminus A)$, in
particular for $\mathcal H$-a.e. $x\in A$.
It follows that $\mathcal H(A\cap \Sigma_{\beta}E\setminus N)=0$,
and thus in fact $\mathcal H(A\cap \Sigma_{\beta}E)=0$.
\end{proof}

Now we can prove the following removability result for BV functions.

\begin{corollary}\label{cor:BV}
Suppose $A\subset X$ is a closed set such that
\[
\liminf_{r\to 0}\frac{\rcapa_1(A\cap B(x,r),B(x,2r))}{\rcapa_1(B(x,r),B(x,2r))}<c_*
\quad\textrm{for 1-q.e. }x\in A.
\]
Then $A$ is isometrically removable for $\BV(X)$.
\end{corollary}
\begin{proof}
Let $u\in \BV(X\setminus A)$.
By Lemma \ref{lem:zero measure}, we have $\mu(A)=0$,
and so $\Vert u\Vert_{L^1(X)}=\Vert u\Vert_{L^1(X\setminus A)}$.
Using the coarea formula \eqref{eq:coarea}, and noting that some integrals below are upper integrals
because measurability is not clear, we estimate
\begin{align*}
		\Vert Du\Vert(X\setminus A)
		&= \int_{-\infty}^{\infty}P(\{u>t\}, X\setminus A)\,dt\\
		&\ge \alpha \int_{(-\infty,\infty)}^*\mathcal H(\partial^*\{u>t\}\cap (X\setminus A))\,dt
		\quad\textrm{by }\eqref{eq:def of theta}\\
		&\ge \alpha\int_{(-\infty,\infty)}^*\mathcal H(\Sigma_{\beta}\{u>t\})\,dt\quad
		\textrm{by Proposition \ref{prop:finite perimeter removability}}\\
		&\ge \alpha C_d^{-1}\int_{(-\infty,\infty)}^*P(\{u>t\},X)\,dt\quad\textrm{by }
		\eqref{eq:def of theta}\\
		&= \alpha C_d^{-1}\Vert Du\Vert(X).
\end{align*}
Thus $u\in\BV(X)$.
Note also that from the above estimates, we obtain that
$\mathcal H(\partial^*\{u>t\}\setminus A)<\infty$ for a.e. $t\in\R$.
By again applying the coarea formula \eqref{eq:coarea}, we get
\begin{align*}
\Vert Du\Vert(A)
&=\int_{-\infty}^{\infty}P(\{u>t\},A)\,dt\\
&\le C_d \int_{-\infty}^{\infty}\mathcal H(\Sigma_{\beta}\{u>t\}\cap A)\,dt
\quad\textrm{by }\eqref{eq:def of theta}\\
&=0
\end{align*}
by Proposition \ref{prop:finite perimeter removability}.
Thus in fact $\Vert Du\Vert(X\setminus A)=\Vert Du\Vert(X)$, meaning that
$A$ is isometrically removable for $\BV(X)$.
\end{proof}

\begin{remark}
Note that here we assumed $A$ to be closed, which made the proof rather straightforward.
As mentioned before, the class of BV functions is only defined in open sets, unlike Newton--Sobolev functions
which are natural to define in any $\mu$-measurable set.
One could consider suitable definitions of the BV class also in non-open sets, but we choose not to
go in this direction here.
\end{remark} 

Now we prove the following theorem, which involves (an improved version of) the case $p=1$ of our
main Theorem \ref{thm:removability theorem}.

\begin{theorem}\label{thm:removability in text}
	Suppose $A\subset X$ is such that
	\[
	\liminf_{r\to 0}\frac{\rcapa_1(A\cap B(x,r),B(x,2r))}{\rcapa_1(B(x,r),B(x,2r))}<c_*
	\quad\textrm{for 1-q.e. }x\in X.
	\]
	Then $A$ isometrically removable for $N^{1,1}(X)$.
\end{theorem}

\begin{proof}
	Note that $\mu(A)=0$ by Lemma \ref{lem:zero measure}.
	First assume that $X\setminus A$ is $1$-finely open.
	
	Let $u\in N^{1,1}(X\setminus A)$.
	First suppose that $-M\le u\le M$ for some $M> 0$.
	Since $X\setminus A$ is $1$-quasiopen by \eqref{eq:finely and quasiopen},
	we find an open set $G\subset X$ such that $(X\setminus A)\cup G$
	is open and thus $A\setminus G$ is closed, and $\capa_1(G)$
	can be chosen arbitrarily small.
	Let $\eps>0$.
	Using Lemma \ref{lem:capacity and Newtonian function}, in fact we also find an open set
	$U\supset G$ such that $\capa_1(U)<\eps$,
	and a function $\eta\in N_0^{1,1}(U)$ such that $0\le \eta\le 1$ on $X$, $\eta=1$ in $G$, and
	$\Vert \eta\Vert_{N^{1,1}(X)}<\eps$.
	Define
	\[
	v:=(1-\eta)u.
	\]
	We first observe that
	\begin{equation}\label{eq:initial Leibniz}
	v\in N^{1,1}(X\setminus A)\quad\textrm{with}\quad
	g_{v,X\setminus A}\le g_{u,X\setminus A}+Mg_{\eta,X}
	\end{equation}
	by the Leibniz rule \eqref{eq:Leibniz rule};
	we understand these minimal $1$-weak upper gradients to be pointwise defined.
	Consider a curve $\gamma\colon [0,\ell_{\gamma}]\to (X\setminus A)\cup G$;
	recall that the latter set is open.
	Excluding a
	$\Mod_1$-negligible family, we can assume that the pair $(v,g_{v,X\setminus A})$
	satisfies the upper gradient inequality
	on all subcurves of $\gamma$ that are in $X\setminus A$; see \cite[Lemma 1.34(c)]{BB}.
	By \eqref{eq:path open} and \eqref{eq:finely and quasiopen},
	excluding another $\Mod_1$-negligible family,
	we can assume that $\gamma^{-1}(X\setminus A)$ 
	and $\gamma^{-1}(G)$ are relatively open sets.
	
	Now $[0,\ell_{\gamma}]$ is a compact set that is covered by the two relatively
	open sets $\gamma^{-1}(X\setminus A)$ and $\gamma^{-1}(G)$.
	By the Lebesgue number lemma,
	there exists a number $\delta>0$ such that every subinterval of $[0,\ell_{\gamma}]$ with
	length at most $\delta$ is contained either in $\gamma^{-1}(X\setminus A)$ or in
	$\gamma^{-1}(G)$.
	Choose $m\in\N$ such that $\ell_{\gamma}/m\le \delta$ and consider the
	subintervals $I_j:=[j\ell_{\gamma}/m,(j+1)\ell_{\gamma}/m]$, $j=0,\ldots,m-1$.
	If $I_j\subset \gamma^{-1}(X\setminus A)$, then by our assumptions on $\gamma$,
	\[
	|v(j\ell_{\gamma}/m)-v((j+1)\ell_{\gamma}/m)|\le \int_{j\ell_{\gamma}/m}^{(j+1)\ell_{\gamma}/m}g_{v,X\setminus A}(\gamma(s))\,ds.
	\]
	Otherwise $I_j\subset \gamma^{-1}(G)$. Recall that
	$v=0$ in $G$. Then
	\begin{align*}
		|v(j\ell_{\gamma}/m)-v((j+1)\ell_{\gamma}/m)|=0.
	\end{align*}
	Adding up the inequalities for $j=0,\ldots,m-1$,
	by \eqref{eq:initial Leibniz}
	we conclude that the
	upper gradient inequality holds for the
	pair $(v,g_{u,X\setminus A}+Mg_{\eta,X})$ on the curve $\gamma$, that is,
	\[
	|v(\gamma(0))-v(\gamma(\ell_{\gamma}))|\le \int_{\gamma}(g_{u,X\setminus A}+Mg_{\eta,X})\,ds.
	\]
	We conclude that $v\in N^{1,1}((X\setminus A)\cup G)$ with
	$g_{v,(X\setminus A)\cup G}\le g_{u,X\setminus A}+Mg_{\eta,X}$,
	and so (we drop the sets from the subscripts)
	\begin{equation}\label{eq:gv and gu}
	\int_{(X\setminus A)\cup G}g_v\,d\mu
	\le \int_{X\setminus A}g_u\,d\mu
	+M\int_X g_{\eta}\,d\mu
	\le \int_{X\setminus A}g_u\,d\mu+M\eps.
	\end{equation}
	Using the coarea formula \eqref{eq:coarea}, we estimate
	\begin{equation}\label{eq:gv Dv estimate}
	\begin{split}
		\int_{(X\setminus A)\cup G}g_{v}\,d\mu
		&\ge \Vert Dv\Vert((X\setminus A)\cup G)\\
		&= \int_{-\infty}^{\infty}P(\{v>t\}, (X\setminus A)\cup G)\,dt\\
		&\ge \alpha \int_{(-\infty,\infty)}^*\mathcal H(\partial^*\{v>t\}\cap [(X\setminus A)\cup G])\,dt\quad\textrm{by }\eqref{eq:def of theta}\\
		&\ge \alpha\int_{(-\infty,\infty)}^*\mathcal H(\Sigma_{\beta}\{v>t\})\,dt\quad\textrm{by Proposition }\ref{prop:finite perimeter removability}\\
		&\ge \alpha C_d^{-1}\int_{(-\infty,\infty)}^*P(\{v>t\},X)\,dt\quad\textrm{by }
		\eqref{eq:perimeter and strong boundary}\\
		&= \alpha C_d^{-1}\Vert Dv\Vert(X).
	\end{split}
	\end{equation}
	We obtain $v\in \BV(X)$. Since $v\in N^{1,1}((X\setminus A)\cup G)$,
	obviously $\Vert Dv\Vert$ is absolutely continuous with respect to
	$\mu$ in the open set $(X\setminus A)\cup G$.
	But we also have by the coarea formula \eqref{eq:coarea}
	and \eqref{eq:def of theta} that
	\[
	\Vert Dv\Vert(A\setminus G)
	\le C_d\int_{-\infty}^{\infty}\mathcal H(\Sigma_{\beta}\{v>t\}\cap (A\setminus G))\,dt=0
	\]
	by Proposition \ref{prop:finite perimeter removability},
	and so $\Vert Dv\Vert$ is absolutely continuous with respect to
	$\mu$ in the entire $X$.
	By \eqref{eq:from BV to Sobolev}, we get $v^*\in N^{1,1}(X)$ with
	\begin{align*}
		\Vert Dv\Vert(X)
		\ge \frac{1}{C}\int_{X} g_{v^*}\,d\mu
		\ge \frac{1}{C}\int_{X\setminus \overline{U}^1} g_{v^*,X\setminus \overline{U}^1}\,d\mu
		= \frac{1}{C}\int_{X\setminus \overline{U}^1} g_{u^*,X\setminus \overline{U}^1}\,d\mu,
	\end{align*}
	where the last equality follows from the fact that $v^*=u^*$ everywhere in
	$X\setminus \overline{U}^1$
	by \eqref{eq:fine and measure topology}.
	Combining this with \eqref{eq:gv Dv estimate} and \eqref{eq:gv and gu}, we get
	\[
	\int_{X\setminus A}g_u\,d\mu+M\eps
	\ge \frac{\alpha C_d^{-1}}{C}\int_{X\setminus \overline{U}^1} 
	g_{u^*,X\setminus \overline{U}^1}\,d\mu.
	\]
	Note that by \eqref{eq:Lebesgue point result}, $v$ has Lebesgue points $1$-q.e.
	in the open set $(X\setminus A)\cup G$. Thus clearly
	\[
	u^*=v^*=v=u\quad 1\textrm{-q.e. in }(X\setminus A)\setminus \overline{U}^1.
	\]
	Note that $\capa_1(\overline{U}^1)<\eps$ by \eqref{eq:capa of fine closure}.
	Taking the limit as $\eps\to 0$, by Lemma \ref{lem:capacities} we get
	$u^*\in N^{1,1}(X)$ with
	\begin{equation}\label{eq:bounded conclusion}
	\int_{X\setminus A}g_u\,d\mu
	\ge \frac{\alpha C_d^{-1}}{C}\int_{X} g_{u^*}\,d\mu,
	\end{equation}
	and
	\begin{equation}\label{eq:Lebesgue points for bounded}
	u^*=u\quad 1\textrm{-q.e. in }X\setminus A.
	\end{equation}
	
	Next consider a general (possibly unbounded) $u\in N^{1,1}(X\setminus A)$.
	We have for the truncations $u_M:=\max\{-M,\min\{M,u\}\}$ that
	\[
	\int_{X\setminus A}g_{u}\,d\mu
	\ge \int_{X\setminus A}g_{u_M}\,d\mu
	\ge \frac{\alpha C_d^{-1}}{C}\int_X g_{(u_M)^*}\,d\mu
	\quad\textrm{by }\eqref{eq:bounded conclusion}.
	\]
	Note that each $(u_M)^*$ has Lebesgue points $1$-q.e. in $X$ by \eqref{eq:Lebesgue point result}.
	Consider a point $x\in X$ that is a Lebesgue point for all $(u_M)^*$.
	Then $(u_M)^*(x)=((u_{M+1})^*)_M(x)$ for all $M\in \N$.
	Moreover, $g_{(u_M)^*}$ is an a.e. increasing sequence by \eqref{eq:truncation} and
	\eqref{eq:quasieverywhere equivalence class}.
	Letting $M\to \infty$, by \cite[Proposition 2.4]{BB} we get for $v:=\limsup_{M\to\infty} (u_M)^*$
	(note that the limit exists $1$-q.e.)
	that
	$v\in N^{1,1}(X)$ and
	\[
	\int_{X\setminus A}g_{u}\,d\mu\ge \frac{\alpha C_d^{-1}}{C}\int_{X} g_{v}\,d\mu.
	\]
	Moreover, by \eqref{eq:Lebesgue points for bounded} we get
	$(u_M)^*=u_M$ $1$-q.e. in $X\setminus A$, and thus
	\[
	v=u\quad 1\textrm{-q.e. in }X\setminus A,
	\]
	and in particular $v=u$ a.e. in $X$.
	By \eqref{eq:ug in finely open} and \eqref{eq:ae to qe},
	\[
	g_{v}=g_{v,X\setminus A}=g_{u,X\setminus A}\quad \textrm{a.e. in }X.
	\]
	Thus $A$ is isometrically removable for $N^{1,1}(X)$.
	
	Finally we remove the assumption that $X\setminus A$ is $1$-finely open.
	Let $A\subset X$ be as in the statement of the theorem, and let
	$u\in N^{1,1}(X\setminus A)$.
	The set $X\setminus \overline{A}^1$ is $1$-finely open, and by \eqref{eq:capa of fine closure}, we have
	\[
	\liminf_{r\to 0}\frac{\rcapa_1(\overline{A}^1\cap B(x,r),B(x,2r))}{\rcapa_1(B(x,r),B(x,2r))}<c_*
	\]
	for $1$-q.e. $x\in X$.
	We of course have $u\in N^{1,1}(X\setminus \overline{A}^1)$, and note that
	$\mu(\overline{A}^1)=0$ by Lemma \ref{lem:zero measure}.
	By the first part of the proof,
	we find $v\in N^{1,1}(X)$ with
	$v=u$ a.e. in $X$ and
	$g_{v}=g_{u,X\setminus \overline{A}^1}$ a.e. in $X$.
	Formula \eqref{eq:ae to qe} and the fact that $A\subset \overline{A}^1$ imply that necessarily
	$g_{v}\ge g_{v,X\setminus A}= g_{u,X\setminus A}\ge g_{u,X\setminus \overline{A}^1}$ a.e. in $X$, and so
	we also get
	$g_{v}=g_{u,X\setminus A}$ a.e. in $X$.
	Thus $A$ is isometrically removable for $N^{1,1}(X)$.
\end{proof}

As mentioned in the introduction, there is a well known equivalence between a set being removable
for Newton--Sobolev functions, and being removable for the Poincar\'e inequality.
We will use a version of this equivalence given in the following two theorems
from the recent paper \cite{BBL}.
The first of these theorems is given by \cite[Theorem 5.4]{BBL} implication $(c)\Rightarrow (f)$.

\begin{theorem}\label{thm:equivalence}
	Let $A\subset X$ be such that $\mu(A)=0$.
	If the set $A$ is isometrically removable for $N^{1,1}(X)$,
	then the space $X\setminus A$ supports a $(1,1)$--Poincar\'e inequality.
\end{theorem} 

Recall the definition from \eqref{eq:poincare inequality minus A}.
The second theorem is given by \cite[Theorem 5.4]{BBL} implication $(f)\Rightarrow (a)$.

\begin{theorem}\label{thm:Bjorn removability}
	Let $A\subset X$ be such that $\mu(A)=0$, and let $1\le p<\infty$.
	Suppose that the space $X\setminus A$ supports a $(1,p)$--Poincar\'e inequality.
	Then the set $A$ is removable for $N^{1,p}(X)$.
\end{theorem}

\begin{proof}[Proof of Theorem \ref{thm:removability theorem}]
By Theorem \ref{thm:removability in text}, the set
$A$ is isometrically removable for $N^{1,1}(X)$.
Then by Theorem \ref{thm:equivalence},
$X\setminus A$ supports a $(1,1)$--Poincar\'e inequality.
Then by H\"older's inequality, $X\setminus A$ also supports a $(1,p)$--Poincar\'e inequality.

Finally by Theorem \ref{thm:Bjorn removability}, the set $A$ is also removable for $N^{1,p}(X)$
for all $1<p<\infty$.
\end{proof}

\section{Discussion}

In this final section,
we make remarks and in particular compare our results with the existing literature.

First we comment on the role of the new Federer-type characterization
of sets of finite perimeter given in Theorem \ref{thm:new Federer}.
We observe that if we were relying on the ordinary Federer's characterization
(explained right before Theorem \ref{thm:new Federer}),
then in \eqref{eq:beta over two} we would need to be considering a point $x\in \partial^*E$
instead of $x\in \Sigma_{\beta}E$, and then there would be no guarantee even that
\[
\limsup_{j\to\infty}\frac{\mu( B(x,r_j/\lambda)\setminus (E\cup A(x,r_j)))}{\mu(B(x,r_j))}>0,
\]
and so the rest of the proof would not work.
In order to make it work,
instead of the capacitary density condition \eqref{eq:liminf intro}
we would need to consider the stronger condition
\begin{equation}\label{eq:stronger density}
\lim_{r\to 0}\frac{\rcapa_1(A\cap B(x,r),B(x,2r))}{\rcapa_1(B(x,r),B(x,2r))}
=0\quad\textrm{for 1-q.e. }x\in X.
\end{equation}
However, this would mean that $\capa_1(b_1 A)=0$ (recall Definition \ref{def:1 fine topology}),
and then it would follow that $\capa_1(A)=0$, by \eqref{eq:A minus b 1 A}.
By \eqref{eq:cap and mod}, $1$-a.e. curve then has empty intersection with $A$, and so
$A$ is clearly isometrically removable for $N^{1,1}(X)$.
Thus this situation is actually very straightforward, owing to the fact that
\eqref{eq:stronger density} is a strong condition that forces $A$ to be very small.
But replacing ``$\lim$'' by  ``$\liminf$''
and $0$ by a small constant $c_*$,
as we do in \eqref{eq:liminf intro}, makes the situation much more intricate.

Now we compare our results with the literature.
According to Theorem A of Koskela--Shanmu\-galingam--Tuo\-mi\-nen \cite{KST},
in a metric space $(X,d,\mu)$ with $\mu$ doubling and supporting a $(1,p)$--Poincar\'e
inequality with $1<p<\infty$,
and with $X$ satisfying a suitable connectivity condition,
if $A\subset X$ is a compact set
satisfying the $t$-porosity condition
\begin{equation}\label{eq:porosity}
\textrm{for every }x\in A,\ A\cap (B(x,tr)\setminus B(x,r))=\emptyset
\ \textrm{ for arbitrarily small }r>0
\end{equation}
with a sufficiently large $t>1$,
then $X\setminus A$ also supports a $(1,p)$--Poincar\'e inequality.
In Theorem B of \cite{KST} (and in the discussion after the theorem)
it is then noted that the aforementioned connectivity condition holds in particular
if $X$ is complete and Ahlfors $Q$-regular (recall \eqref{eq:Ahlfors Q regular}), with $Q>1$, and $1<p\le Q$.
Korte \cite[Theorem 3.3]{Kor} shows that the connectivity condition
can in fact be obtained only assuming the upper mass bound
\begin{equation}\label{eq:upper mass bound}
\frac{\mu(B(x,r))}{\mu(B(x,R))}\le C_0\left(\frac{r}{R}\right)^Q
\end{equation}
for all $x\in X$ and $0<r\le R<\diam X$,
and constants $Q>1$, $C_0>0$.
(Recall that we considered this type of condition already in \eqref{eq:upper bound with sigma}.)
Now we get the following extension of the results of \cite{KST} to the case $p=1$.
Recall that for us, completeness, doubling, and $(1,1)$--Poincar\'e are standing assumptions on the space.

\begin{corollary}\label{cor:removability}
	Assume that the upper mass bound \eqref{eq:upper mass bound} holds
	for all $x\in X$ and $0<r\le R<\diam X$, and with
	$Q>1$. Suppose that a closed set $A\subset X$ is $t$-porous
	with sufficiently large $t>1$,
	depending only on $C_d,C_P,\lambda,C_0$, and $Q$.
	Then $X\setminus A$ supports a $(1,1)$--Poincar\'e inequality,
	and $A$ is removable for $N^{1,p}(X)$, with $1\le p<\infty$.
\end{corollary}

\begin{proof}
Let $x\in A$. Using the $t$-porosity, we find a sequence $r_j\searrow 0$ such that
$A\cap (B(x,tr_j)\setminus B(x,r_j))=\emptyset$ for every $j\in\N$.
Thus
\begin{align*}
	\liminf_{r\to 0}\frac{\rcapa_1(A\cap B(x,r),B(x,2r))}{\rcapa_1(B(x,r),B(x,2r))}
	&\le \liminf_{j\to\infty}\frac{\rcapa_1(B(x,r_j),B(x,2tr_j))}{\rcapa_1(B(x,tr_j),B(x,2tr_j))}\\
	&\le C_{\textrm{sob}}C_d \liminf_{j\to\infty}\frac{\mu(B(x,r_j))/r_j}{\mu(B(x,tr_j))/(tr_j)}
	\quad\textrm{by }\eqref{eq:capa basic estimate}\\
	&\le C_{\textrm{sob}}C_d C_0 t\liminf_{j\to\infty}\left(\frac{r_j}{tr_j}\right)^{Q}
	\quad\textrm{by }\eqref{eq:upper mass bound}\\
	&= C_{\textrm{sob}} C_d C_0  t^{1-Q}.
\end{align*}
In fact this holds for every $x\in X$, since $A$ is closed and so 
for $x\in X\setminus A$ the limit is trivially zero. 
Now, $C_{\textrm{sob}} C_d C_0 t^{1-Q}< c_*$ for large enough $t>1$, and so
the capacitary density condition of Theorem \ref{thm:removability theorem} holds,
and this implies both claims.
\end{proof}

The proofs of Theorems A and B in \cite{KST} are based on a construction where
the $(1,p)$--Poincar\'e inequality is applied to chains of balls situated in the annuli
that do not intersect $A$.
The advantage of our techniques is of course that we only require the capacitary density condition,
instead of the strong assumption of very large annuli having empty intersection with $A$.
Thus Corollary \ref{cor:removability} gives an improved removability result also in the
case $1<p<\infty$,
though only in spaces that support a $(1,1)$--Poincar\'e inequality.

The porosity condition is a very geometric condition.
In closing, we examine the possibility of expressing
the capacitary density condition also in a
geometric form, using the Hausdorff content \eqref{eq:Hausdorff content}.
The comparability between the 1-capacity and the Hausdorff
content is generally a well known fact, and so we only sketch the proof.

Consider the upper mass bound
\begin{equation}\label{eq:lower mass bound}
	\frac{\mu(B(x,r))}{\mu(B(x,R))}\le C_0\frac{r}{R}
	\quad\textrm{for all }x\in X\textrm{ and }0<r\le 4R<\diam X,
\end{equation}
and for some constant $C_0>0$.
This is a fairly mild requirement, since roughly speaking it says that the space has dimension
at least $1$ at every point $x$.
\begin{proposition}
Suppose the upper mass bound \eqref{eq:lower mass bound} holds, and let $A\subset X$.
Then for every $x\in X$, we have
\begin{align*}
\frac{1}{C}\liminf_{r\to 0}\frac{\rcapa_1(A\cap B(x,r),B(x,2r))}{\rcapa_1(B(x,r),B(x,2r))}
&\le \liminf_{r\to 0}\frac{\mathcal H_\infty(A\cap B(x,r))}{\mu(B(x,r))/r}\\
&\le C\liminf_{r\to 0}\frac{\rcapa_1(A\cap B(x,r),B(x,2r))}{\rcapa_1(B(x,r),B(x,2r))}
\end{align*}
for a constant $C\ge 1$ depending only on $C_d,C_P,\lambda,C_0$.

\end{proposition}

\begin{proof}
By \eqref{eq:capa basic estimate}, the denominators are comparable for small $r$.
To prove the first inequality,
let $\eps>0$ and
consider a covering $\{B_j=B(x_j,r_j)\}_{j}$ of $A\cap B(x,r)$ for which
$\sum_{j}\frac{\mu(B_j)}{r_j}<\mathcal H_\infty(A\cap B(x,r))+\eps$.
We can assume that each ball $B_j$ intersects $A\cap B(x,r)$.
If $r_j\ge r/4$ for some $j$, then  we have 
\[
\frac{\mu(B_j)}{r_j}
\ge \frac{1}{C_d^4}\frac{\mu(B(x,r_j))}{r_j}
\ge \frac{1}{C_0C_d^4}\frac{\mu(B(x,r))}{r}
\quad\textrm{by }\eqref{eq:lower mass bound},
\]
and so
\[
\frac{\rcapa_1(A\cap B(x,r),B(x,2r))}{\rcapa_1(B(x,r),B(x,2r))}
\le 1 \le C_0C_d^4\frac{\mathcal H_\infty(A\cap B(x,r))+\eps}{\mu(B(x,r))/r}.
\]
Thus we can assume $r_j\le r/4$ for all $j$. Then
\begin{align*}
\rcapa_1(A\cap B(x,r),B(x,2r))
&\le \sum_{j}\rcapa_1(B_j,B(x,2r))\\
&\le \sum_{j}\rcapa_1(B_j,2B_j)\quad\textrm{since }2B_j\subset B(x,2r)\\
&\le C_d\sum_{j}\frac{\mu(B_j)}{r_j}\quad\textrm{by }\eqref{eq:capa basic estimate}\\
&\le C_d (\mathcal H_{\infty}(A\cap B(x,r))+\eps).
\end{align*}
Letting $\eps\to 0$, we get the first inequality.

To prove the second inequality,
let $\eps>0$.
Consider a set $A(x,r)\subset B(x,2r)$
with $A\cap B(x,r)\subset A(x,r)\subset B(x,2r)$
and $P(A(x,r),X)\le \rcapa_1(A\cap B(x,r),B(x,2r))+\eps$,
just as on page \pageref{Axr}.
If $r>0$ is small, then applying \cite[Theorem 3.1]{KKST}, we find balls $\{B_j=B(x_j,r_j)\}_j$
covering the interior of $A(x,r)$, with
\[
\mathcal H_{\infty}(A\cap B(x,r))
\le \sum_j \frac{\mu(B_j)}{r_j}
\le CP(A(x,r),X)
\le C(\rcapa_1(A\cap B(x,r),B(x,2r))+\eps)
\]
for a constant $C$ depending only on $C_d,C_P,\lambda$.
Letting $\eps\to 0$, we get the second inequality.
\end{proof}

\end{document}